\documentclass[runningheads,a4paper]{llncs}
\usepackage{amssymb}
\usepackage{graphicx}
\usepackage{url}
\usepackage{enumerate,lineno,setspace,float}

\urldef{\mailsa}\path|devsi.bantva@gmail.com|
\newcommand{\keywords}[1]{\par\addvspace\baselineskip
\noindent\keywordname\enspace\ignorespaces#1}

\def\ep{\epsilon}
\def\p{\prime}
\def\({\left(}
\def\){\right)}
\def\[({\left[}
\def\]{\right]}
\def\d{{\rm d}}
\def\diam{\rm diam}
\def\dis{\displaystyle}

\begin{document}
\mainmatter  

\title{On a lower bound for the eccentric connectivity index of graphs}

\author{Devsi Bantva}
\institute{Lukhdhirji Engineering College, Morvi - 363 642 \\
Gujarat (INDIA) \\
\mailsa\\}

\toctitle{On a lower bound for the eccentric connectivity index of graphs}
\tocauthor{Devsi Bantva}
\maketitle

\begin{abstract}
The eccentric connectivity index of a graph $G$, denoted by $\xi^{c}(G)$, defined as $\xi^{c}(G)$ = $\sum_{v \in V(G)}\epsilon(v) \cdot \d(v)$, where $\epsilon(v)$ and $\d(v)$ denotes the eccentricity and degree of a vertex $v$ in a graph $G$, respectively. The volcano graph $V_{n,d}$ is a graph obtained from a path $P_{d+1}$ and a set $S$ of $n-d-1$ vertices, by joining each vertex in $S$ to a central vertex/vertices of $P_{d+1}$. In \cite{Morgan2}, Morgan \emph{et al.} proved that $\xi^{c}(G) \geq \xi^{c}(V_{n,d})$ for any graph of order $n$ and diameter $d \geq 3$. In this paper, we present a short and simple proof of this result by considering the adjacency of vertices in graphs.

\end{abstract}

\section{Introduction}
Let $G$ be a finite, connected and undirected graph without loops and multiple edges. We denote the \emph{vertex set} of $G$ by $V(G)$. The \emph{distance} between two vertices $u$ and $v$ of $G$, denoted by $d(u,v)$, is the least length of a $u,v-$path in $G$. The \emph{eccentricity} of a vertex $v$ in a graph $G$, denoted by $\epsilon(v)$, is the distance of a vertex farthest from $v$ in $G$. The \emph{degree} of a vertex $v$ in a graph $G$, denoted by $\d(v)$, is the number of edges incident to it.

A topological index is a numerical graph invariants used for Quantitative Structure-Activity Relationship (QSAR) and Quantitative Structure-Property Relationship (QSPR) studies. The Wiener index, introduced in 1947 by Herold Wiener\cite{Wiener}, is the first non-trivial topological index in Chemistry. Then after many topological indices have been defined such as Zagreb index, PI-index etc. and successfully used to study the chemical, pharmaceutical and other properties of molecules. More recently, a new adjacent-cum-distance based topological index, the \emph{eccentric connectivity index} (or ECI for short) denoted by $\xi^{c}(G)$ has been introduced by Sharma, Goswami and Madan\cite{Sharma} which is defined as
\begin{equation}\label{def:eci}
\xi^{c}(G) = \displaystyle\sum_{v \in V(G)} \epsilon(v) \cdot \d(v)
\end{equation}
The eccentric connectivity index has been employed successfully for the development of numerous mathematical models for the prediction of biological activities of diverse nature. Recently, many results on the eccentric connectivity index have been obtained and some of them have been applied as means for modeling chemical, pharmaceutical and other properties of molecules, for details see \cite{Gupta,Ilic,Sardana1,Sardana2,Sharma,Zhou}.

The common trend of research for the topological indices and its variants is to determine the extremal graphs for the given topological index or its variant. Also the trend is to determine the extremal trees for the given topological index or its variant. For most of introduced topological indices or its variants, the extremal graphs or trees are determined except Wiener index; the origin of all topological indices and its variants. The same approach is considered by Morgan \emph{et al.} for the eccentric connectivity index of graphs in \cite{Morgan1} and \cite{Morgan2}.

In \cite{Morgan1}, Morgan \emph{et al.} noted the eccentric connectivity index of some basic graph families and determined the eccentric connectivity index of other three classes of graphs namely broom graph $B_{n,d}$ (a graph which consists a path $P_{d}$, together with ($n-d$) end vertices all adjacent to the same end vertex of $P_{d}$), lollipop graph $L_{n,d}$ (a graph obtained from a complete graph $K_{n-d}$ and a path $P_{d}$, by joining one of the end vertices of $P_{d}$ to all the vertices of $K_{n-d}$) and volcano graph $V_{n,d}$ (a graph obtained from a path $P_{d+1}$ and a set $S$ of $n-d-1$ vertices, by joining each vertex in $S$ to a central vertex/vertices of $P_{d+1}$). Note that for a fixed value of $n$, when $d$ is even, the volcano graph $V_{n,d}$ is unique; whereas when $d$ is odd, there may be several non-isomorphic volcano graphs $V_{n,d}$. The readers are advised to see \cite{Morgan1} and \cite{Morgan2} for figures and details on these graph families. The eccentric connectivity index of path $P_{d+1}$ and volcano graph $V_{n,d}$ is given as follows.
\begin{proposition} Let $n \geq 2$ be an integer. Then
\begin{eqnarray}\label{eci:pn}
\xi^{c}(P_{n}) = & \left\{
\begin{array}{ll}
\frac{1}{2}(3n^{2}-6n+4), & \mbox{ for $n$ even}, \\ [0.3cm]
\frac{3}{2}(n-1)^{2}, & \mbox{ for $n$ odd}.
\end{array}
\right.
\end{eqnarray}
\end{proposition}
\begin{proposition} Let $n,d$ be non-negative integers. Then
\begin{eqnarray}\label{eci:vnd}
\xi^{c}(V_{n,d}) = & \left\{
\begin{array}{ll}
nd+n+\frac{d^{2}}{2}-2d-1, & \mbox{ for $d$ even}, \\ [0.3cm]
nd+2n+\frac{d^{2}}{2}-3d-\frac{3}{2}, & \mbox{ for $d$ odd}.
\end{array}
\right.
\end{eqnarray}
\end{proposition}

In \cite{Morgan1}, Morgan \emph{et al.} gave a lower bound for the eccentric connectivity index of trees and proved that $\xi^{c}(T) \geq \xi^{c}(V_{n,d})$ for any tree $T$ of order $n \geq 3$ and diameter $d$. Later, in \cite{Morgan2}, Morgan \emph{et al.} extended this lower bound for the eccentric connectivity index to arbitrary connected graph of order $n$ and diameter $d \geq 3$ but we emphasize that the proof is lengthy and complicated. In this paper, we give a short proof of this result by considering the connectivity of vertices in graph which is a very simple approach. Moreover, this approach can also be extend to prove similar types of results for other topological indices.

\section{Preliminaries}

We follow \cite{West} for graph theoretic definition and notation. A tree $T$ is a connected graph that contains no cycle. A caterpillar is a tree in which all the vertices are within distance one from central path. The diameter of a graph $G$, denoted by $\diam(\emph{G})$ or simply $d$, is max\{$d(u,v) : u, v \in V(G)$\}. The \emph{degree} of a vertex in a graph $G$, denoted by $\d(v)$, is defined as the number of edges incident to it. Let $H \subseteq V(G)$ then for any $v \in V(G)$, define $d(v|H)$ is the degree of a vertex $v$ in a subgraph induced by $H$ of $G$. Note that if $H$ = $V(G)$ then $d(v|H)$ = $d(v)$; otherwise $d(v|H) \leq d(v)$. The center of a graph $G$, denoted by $C(G)$, is the set of vertices with minimum eccentricity. Note that for any $v \in V(G), \epsilon(v) \geq \lceil d/2 \rceil$ and $\epsilon(v)$ = $\lceil d/2 \rceil$ if and only if $v \in V(C(G))$. We notice the following well known results about center of graphs.
\begin{proposition} The center $C(G)$ of a graph $G$ is contained in a block of $G$.
\end{proposition}
\begin{proposition} The center $C(T)$ of a tree $T$ consists of a single vertex or two adjacent vertices.
\end{proposition}
The following lemma characterize the vertices with minimum eccentricity in a graph $G$ which is useful for our main result.
\begin{lemma}\label{lem1} Let $G$ be any graph and $P_{d+1} = v_{0}-v_{1}-...-v_{d}$ be a fixed diametral path joining $v_{0}$ and $v_{d}$.
\begin{enumerate}
  \item[(a)] If $\ep(v)$ = $d/2$ for $v \in V(G \setminus P_{d+1})$ then $v$ is on other diametral path joining $v_{0}$ and $v_{d}$, where $C(P_{d+1})$ = \{$w$\}.
  \item[(b)] If $\ep(v)$ = $(d+1)/2$ for $v \in V(G \setminus P_{d+1})$ then either $v$ is on other diametral path joining $v_{0}$ and $v_{d}$ or $v$ is adjacent to both $w$ and $w^{'}$, where $C(P_{d+1})$ = \{$w,w^{'}$\}.
\end{enumerate}
\end{lemma}
\begin{proof}(a) Let $G$ be any graph and $P_{d+1} = v_{0}-v_{1}-...-v_{d}$ be a diametral path joining $v_{0}$ and $v_{d}$ such that $C(P_{d+1})$ = \{$w$\}. Let $v \in V(G \setminus P_{d+1})$ such that $\ep(v)$ = $d/2$. If possible then assume that $v$ is not on any diametral path joining $v_{0}$ and $v_{d}$. Then it is clear that either $d(v,v_{0}) > d/2$ or $d(v,v_{d}) > d/2$; otherwise the shortest path $P^{'} = v_{0} -...- v -...- v_{d}$ is a diametral path as $\ep(v)$ = $d/2$, $d(v_{0},v_{d})$ = $d$ and $\ep(u) \geq d/2$ for any $u \in V(G)$. But note that one of $d(v,v_{0}) > d/2$ or $d(v,v_{d}) > d/2$ gives $\ep(v) > d/2$, a contradiction. Hence $v$ is on diametral path joining $v_{0}$ and $v_{d}$.

(b) Let $G$ be any graph and $P_{d+1} = v_{0}-v_{1}-...-v_{d}$ be a diametral path joining $v_{0}$ and $v_{d}$ such that $C(P_{d+1})$ = \{$w,w^{'}$\}. It is clear that if $v$ is adjacent to both $w$ and $w^{'}$ then $\ep(v)$ = $(d+1)/2$. So assume that $\ep(v)$ = $(d+1)/2$ for some $v \in V(G \setminus P_{d+1})$ and $v$ is not adjacent to $w$ and $w^{'}$ then as in case (a) one can prove that $v$ is on other diametral path joining $v_{0}$ and $v_{d}$.
\end{proof}

\begin{lemma}\label{lem2} Let $G$ be any graph and $v \in V(G)$ such that $\ep(v)$ = $\lceil d/2 \rceil$ then $\d(v) \geq 2$.
\end{lemma}
\begin{proof} By Lemma \ref{lem1} if $\ep(v)$ = $\lceil d/2 \rceil$ then either $v \in C(P_{d+1})$ or $v$ is on other diametral path joining $v_{0}$ and $v_{d}$ where $P_{d+1} = v_{0} - v_{1} -...-v_{d}$ is a diametral path joining $v_{0}$ and $v_{d}$ or $v$ is adjacent to both $w$ and $w^{'}$ in the case when $C(P_{d+1})$ =\{$w,w^{'}$\}. Hence in any case $\d(v) \geq 2$.
\end{proof}

\section{Main Result}
In this section, we continue to use the terminology and notation defined in previous section. First we prove the following Theorem.

\begin{theorem}\label{main1} Let $G$ = $(V,E)$ be a connected graph of order $n$, and diameter $d \geq 3$ such that every spanning tree $T$ of $G$ is a caterpillar of diameter $d$. Then
\begin{equation}\xi^{c}(G) \geq \xi^{c}(V_{n,d}).\end{equation}
\end{theorem}
\begin{proof} Let $G$ be a connected graph of order $n$ and diameter $d \geq 3$ such that every spanning tree $T$ of $G$ is a caterpillar of diameter $d$. Let $P_{d+1}$ = $v_{0}-v_{1}-...-v_{d}$ be a fixed diametral path of a graph $G$. It is clear that a caterpillar $T$ contains $P_{d+1}$ and it is the central path of $T$. Then define \\
$P = \{v \in V(G) : v \in V(P_{d+1})\}$; \\
$P^{\prime} = \{v \in V(G) : v \not\in V(P_{d+1})\}$; \\
$P_{c} = \{v \in P : \epsilon(v) = \lceil d/2 \rceil\}$; \\
$P_{c^{\prime}} = \{v \in P : \epsilon(v) > \lceil d/2 \rceil\}$; \\
$P_{c}^{\prime} = \{v \in P^{\prime} : \epsilon(v) = \lceil d/2 \rceil \}$; \\
$P_{c^{\prime}}^{\prime} = \{v \in P^{\prime} : \epsilon(v) > \lceil d/2 \rceil \}$; \\
$P_{cc}^{\prime} = \{v \in P_{c}^{\p} : v \mbox{ is adjacent to } C(P_{d+1})\}$; \\
$P_{cc^{\prime}}^{\prime} = \{v \in P_{c}^{\p} : v \mbox{ is not adjacent to } C(P_{d+1})\}$; \\
$P_{c^{\prime}c}^{\prime} = \{v \in P_{c^{\p}}^{\p} : v \mbox{ is adjacent to } C(P_{d+1})\}$; \\
$P_{c^{\prime}c^{\prime}}^{\prime} = \{v \in P_{c^{\p}}^{\p} : v \mbox{ is not adjacent to } C(P_{d+1})\}$.

Let $|P_{c}^{\prime}|$ = $n_{1}$ and $|P_{c^{\prime}}^{\prime}|$ = $n_{2}$, where $0 \leq n_{1},n_{2} \leq n-d-1$; $|P_{cc}^{\prime}|$ = $n_{11}$ and $|P_{cc^{\prime}}^{\prime}|$ = $n_{12}$, where $0 \leq n_{11},n_{12} \leq n_{1}$; $|P_{c^{\prime}c}^{\prime}|$ = $n_{21}$ and $|P_{c^{\prime}c^{\prime}}^{\prime}|$ = $n_{22}$, where $0 \leq n_{21},n_{22} \leq n_{2}$. Note that $n_{1}+n_{2}$ = $n-d-1$, $n_{11}+n_{12}$ = $n_{1}$ and $n_{21}+n_{22}$ = $n_{2}$. Moreover, $V(G)$ = $P \cup P^{\prime}$ = $P_{c} \cup P_{c^{\prime}} \cup P_{c}^{\prime} \cup P_{c^{\prime}}^{\prime}$ = $P_{c} \cup P_{c^{\prime}} \cup P_{cc}^{\prime} \cup P_{cc^{\prime}}^{\prime} \cup P_{c^{\prime}c}^{\prime} \cup P_{c^{\prime}c^{\prime}}^{\prime}$.

We consider the following two cases. \\
\textsf{Case 1}: $d$ is even.
\begin{eqnarray*}
\xi^{c}(G) & = & \dis\sum_{v \in v(G)} \ep(v)\d(v) \\
& = & \dis\sum_{v \in P} \ep(v)\d(v) + \dis\sum_{v \in P^{\prime}}  \ep(v)\d(v) \\
& = & \dis\sum_{v \in P} \ep(v) ( \d(v|P)+\d(v|P^{\p}) ) + \dis\sum_{v \in P_{c}^{\p}} \ep(v)\d(v) + \dis\sum_{v \in P_{c^{\p}}^{\p}} \ep(v)\d(v) \\
& = & \dis\sum_{v \in P} \ep(v)\d(v|P) + \dis\sum_{v \in P} \ep(v)\d(v|P^{\p}) + \dis\sum_{v \in P_{c}^{\p}} \ep(v)\d(v) + \dis\sum_{v \in P_{c^{\p}c}^{\p}} \ep(v)\d(v) + \\
& & \dis\sum_{v \in P_{c^{\p}c^{\p}}^{\p}} \ep(v)\d(v) \\
& = & \dis\sum_{v \in P} \ep(v)\d(v|P) + \dis\sum_{v \in P_{c}} \ep(v)\d(v|P^{\p}) + \dis\sum_{v \in P_{c^{\p}}} \ep(v)\d(v|P^{\p}) + \dis\sum_{v \in P_{c}^{\p}} \ep(v)\d(v) + \\
& & \dis\sum_{v \in P_{c^{\p}c}^{\p}} \ep(v)\d(v) + \dis\sum_{v \in P_{c^{\p}c^{\p}}^{\p}} \ep(v)\d(v) \\
& \geq & \frac{3}{2}d^{2}+n_{21}\(\frac{d}{2}\)(1) + (2n_{1}+n_{22})\(\frac{d}{2}+1\)(1) + n_{1}\(\frac{d}{2}\)(2) + \\
& & n_{21}\(\frac{d}{2}+1\)(1) + n_{22}\(\frac{d}{2}+1\)(1) \\
& \geq & \frac{3}{2}d^{2}+n_{21}\(\frac{d}{2}\) + n_{1}\(\frac{d}{2}+1\) + n_{22}\(\frac{d}{2}+1\) + n_{1}\(\frac{d}{2}\) \\
& & + n_{21}\(\frac{d}{2}+1\) + n_{22}\(\frac{d}{2}+1\) \\
& \geq & \frac{3}{2}d^{2}+n_{21}\(\frac{d}{2}\) + n_{1}\(\frac{d}{2}+1\) + n_{22}\(\frac{d}{2}\) + n_{1}\(\frac{d}{2}\) + n_{21}\(\frac{d}{2}+1\) + \\ & & n_{22}\(\frac{d}{2}+1\) \\
& = & \frac{3}{2}d^{2} + n_{1}\(\frac{d}{2}\) + n_{1}\(\frac{d}{2}+1\) + n_{2}\(\frac{d}{2}\) + n_{2}\(\frac{d}{2}+1\) \\
& = & \frac{3}{2}d^{2} + (n_{1}+n_{2})\(\frac{d}{2}\) + (n_{1}+n_{2})\(\frac{d}{2}+1\) \\
& = & \frac{3}{2}d^{2} + (n-d-1)\(\frac{d}{2}\) + (n-d-1)\(\frac{d}{2}+1\) \\
& = & nd + n + \frac{d^{2}}{2}-2d-1 \\
& = & \xi^{c}(V_{n,d}).
\end{eqnarray*}
\textsf{Case 2}: $d$ is odd.
\begin{eqnarray*}
\xi^{c}(G) & = & \dis\sum_{v \in v(G)} \ep(v)\d(v) \\
& = & \dis\sum_{v \in P} \ep(v)\d(v) + \dis\sum_{v \in P^{\prime}}  \ep(v)\d(v) \\
& = & \dis\sum_{v \in P} \ep(v) ( \d(v|P)+\d(v|P^{\p}) ) + \dis\sum_{v \in P_{c}^{\p}} \ep(v)\d(v) + \dis\sum_{v \in P_{c^{\p}}^{\p}} \ep(v)\d(v) \\
& = & \dis\sum_{v \in P} \ep(v)\d(v|P) + \dis\sum_{v \in P} \ep(v)\d(v|P^{\p}) + \dis\sum_{v \in P_{c}^{\p}} \ep(v)\d(v) + \dis\sum_{v \in P_{c^{\p}c}^{\p}} \ep(v)\d(v) + \\
& & \dis\sum_{v \in P_{c^{\p}c^{\p}}^{\p}} \ep(v)\d(v) \\
& = & \dis\sum_{v \in P} \ep(v)\d(v|P) + \dis\sum_{v \in P_{c}} \ep(v)\d(v|P^{\p}) + \dis\sum_{v \in P_{c^{\p}}} \ep(v)\d(v|P^{\p}) + \dis\sum_{v \in P_{c}^{\p}} \ep(v)\d(v) + \\
& & \dis\sum_{v \in P_{c^{\p}c}^{\p}} \ep(v)\d(v) + \dis\sum_{v \in P_{c^{\p}c^{\p}}^{\p}} \ep(v)\d(v) \\
& \geq & \frac{3}{2}d^{2}+\frac{1}{2} + \(2n_{11}+n_{21}\)\(\frac{d+1}{2}\)(1) + \(2n_{12}+n_{22}\)\(\frac{d+3}{2}\)(1) + \\
& & n_{1}\(\frac{d}{2}\)(2) + n_{21} \(\frac{d+3}{2}\)(1) + n_{22}\(\frac{d+3}{2}\)(1) \\
& \geq & \frac{3}{2}d^{2}+\frac{1}{2} + n_{11}\(\frac{d+3}{2}\) + n_{12}\(\frac{d+3}{2}\) + n_{21}\(\frac{d+1}{2}\) + \\
& & n_{22}\(\frac{d+3}{2}\) + n_{1}\(\frac{d+1}{2}\) +\(n_{21}+n_{22}\)\(\frac{d+3}{2}\) \\
& = & \frac{3}{2}d^{2}+\frac{1}{2} + \(n_{11}+n_{12}\)\(\frac{d+3}{2}\) + \(n_{21}+n_{22}\)\(\frac{d+1}{2}\) + \\
& & n_{1}\(\frac{d+1}{2}\) + n_{2}\(\frac{d+3}{2}\) \\
& = & \frac{3}{2}d^{2}+\frac{1}{2} + n_{1}\(\frac{d+3}{2}\) + n_{2}\(\frac{d+1}{2}\) + n_{1}\(\frac{d+1}{2}\) + n_{2}\(\frac{d+3}{2}\) \\
& = & \frac{3}{2}d^{2}+\frac{1}{2} + \(n_{1}+n_{2}\)\(\frac{d+3}{2}\) + \(n_{1}+n_{2}\)\(\frac{d+1}{2}\) \\
& = & \frac{3}{2}d^{2}+\frac{1}{2} + \(n-d-1\)\(\frac{d+3}{2}\) + \(n-d-1\)\(\frac{d+1}{2}\) \\
& = & \frac{3}{2}d^{2}+\frac{1}{2} + \(n-d-1\)\(d+1\) + \(n-d-1\) \\
& = & nd+2n-\frac{d^{2}}{2}-3d-\frac{3}{2} \\
& = & \xi^{c}(V_{n,d}).
\end{eqnarray*}
\end{proof}

\begin{theorem}\label{main2} Let $G$ be a connected graph of order $n$ and diameter $d \geq 2$. Then
\begin{equation} \xi^{c}(G) \geq \xi^{c}(V_{n,d}). \end{equation}
\end{theorem}
\begin{proof} Let $G_{0} \subseteq G_{1} \subseteq \ldots \subseteq G_{k} = G$ be a sequence of subgraphs such that $G_{0}$ is a connected subgraph of $G$ which contain a diametral path $P_{d+1}$ and every spanning tree of $G_{0}$ is a caterpillar of diameter $d$, and $G_{i+1}$ ($0 \leq i \leq k-1$) is a induced subgraph of $G$ with vertex set $V(G_{i+1})$ = $V(G_{i}) \cup \{v\}, v \in V(G \setminus G_{i})$. Let the order of $G_{i}$ is $n_{i}$ then $|G_{k}|$ = $n_{k}$ = $n$. Then by Theorem \ref{main1}, we obtain $\xi^{c}(G_{0}) \geq V_{n_{0},d}$. Now consider the graph $G_{1}$ = $G_{0} \cup \{v\}$ where $v \in V(G \setminus G_{0})$. Note that for a newly added vertex $v$ in $G_{0}$, either $\ep(v) > \lceil d/2 \rceil$ or $\ep(v)$ = $\lceil d/2 \rceil$. If $\ep(v) > \lceil d/2 \rceil$ then it contribute one degree for some vertex of $G_{0}$ and hence $v$ contribute at least $\(\lceil d/2 \rceil\)(1) + \(\lceil d/2 \rceil+1\)(1)$ for $\xi^{c}(G_{1})$ as $G_{1}$ is connected and $\ep(u) \geq d/2$ for every $u \in V(G_{0})$. Hence we obtain, $\xi^{c}(G_{1}) \geq \xi^{c}(G_{0}) + \(\lceil d/2 \rceil\)(1)  + \(\lceil d/2 \rceil+1\)(1) = \xi^{c}(V_{n_{0},d}) + \(\lceil d/2 \rceil\)(1) + \(\lceil d/2 \rceil+1\)(1) = \xi^{c}(V_{n_{1},d})$. If $\ep(v)$ = $\lceil d/2 \rceil$ then by Lemma \ref{lem2}, $\d(v) \geq 2$ and it adjacent to at least two vertices of $G_{0}$. Hence $v$ contribute at least $4 \(\lceil d/2 \rceil\) > \(\lceil d/2 \rceil\) + \(\lceil d/2 \rceil+1\)$  for $\xi^{c}(G_{1})$. Hence we obtain $\xi^{c}(G_{1}) \geq \xi^{c}(G_{0}) + 4\(\lceil d/2 \rceil\) \geq \xi^{c}(V_{n_{0},d}) + \(\lceil d/2 \rceil+1\)$ = $\xi^{c}(V_{n_{1},d})$.

Continuing in this way, finally we obtain $\xi^{c}(G_{k})$ $\geq$ $\xi^{c}(G_{k-1}) + \(\lceil d/2 \rceil+1\)(1) + \(\lceil d/2 \rceil\)(1)$ = $\xi^{c}(V_{n_{k-1},d}) + \(\lceil d/2 \rceil+1\)(1) + \(\lceil d/2 \rceil\)(1)$ = $\xi^{c}(V_{n_{k},d})$ = $\xi^{c}(V_{n,d})$.

Note that equality holds if at each step of above procedure equality holds and hence we obtain that volcano graph $V_{n,d}$ attain a lower bound which completes the proof.
\end{proof}

\begin{corollary} Let $T$ be a tree of order $n$ and diameter $d \geq 3$. Then
$$ \xi^{c}(T) \geq \xi^{c}(V_{n,d}).$$
\end{corollary}

\begin{example} The readers are advised to refer the following example for the procedure used in Theorem \ref{main1} and \ref{main2} to give a lower bound for the eccentric connectivity index of graphs.
\end{example}
In Fig. \ref{Figure1}, the graph $G$ of order 19 and diameter 7 is shown in which $P_{8} = v_{0}-v_{1}-...-v_{7}$ is a fixed diametral path and the vertices with circle are vertices with minimum eccentricity.

\begin{figure}[h!]
\begin{center}
\includegraphics[width=7.5cm]{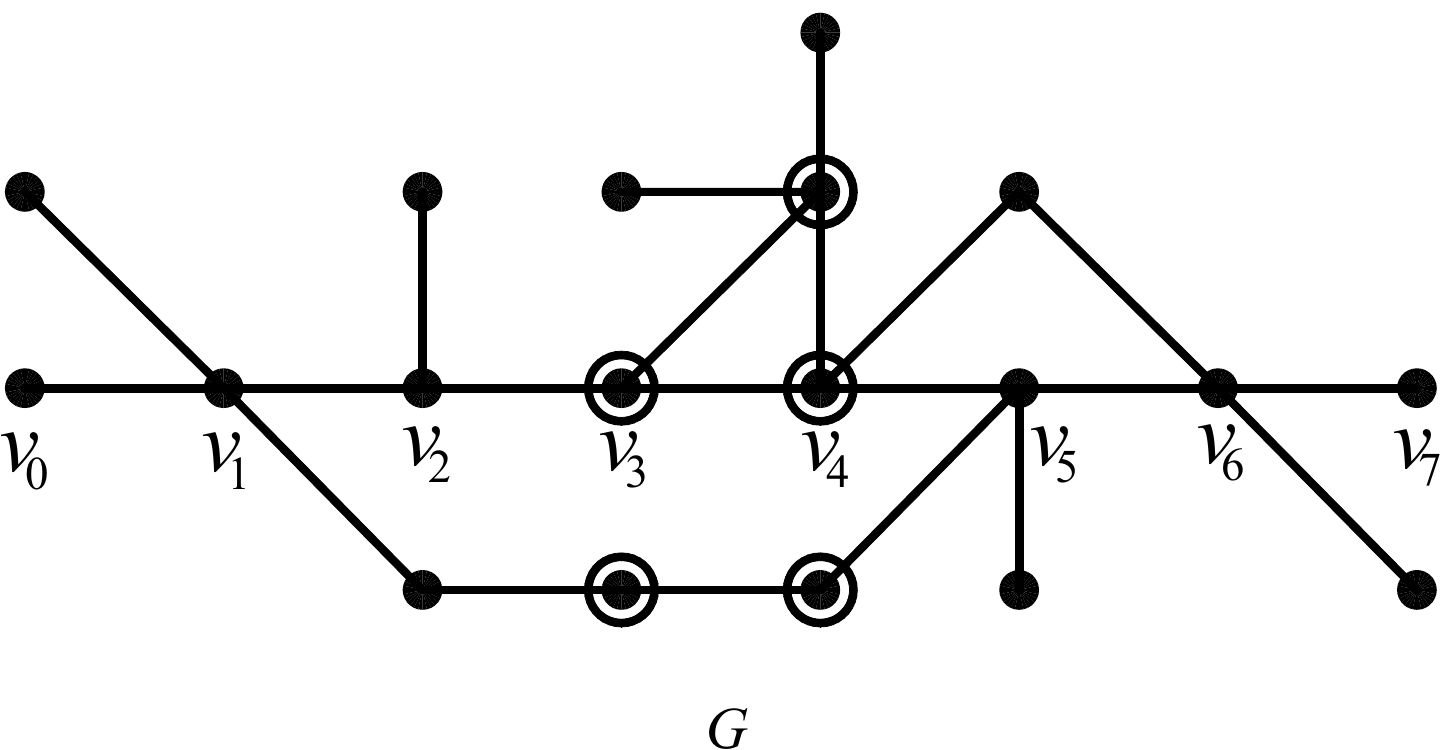}
\caption{Graph $G$ of order 19 and diameter 7.}\label{Figure1}
\end{center}
\end{figure}

In Fig. \ref{Figure2}, a sequence of subgraphs $G_{0} \subset G_{1} \subset G_{2} \subset G_{3} = G$ of $G$ with ordered pair whose first coordinate denote vertex degree and second coordinate denote eccentricity of that vertex in $G_{i}, 0 \leq i \leq 3$ is shown. It is clear that $G_{0}$ is a graph whose each spanning tree is a caterpillar of diameter 7 and $G_{i+1}$ = $G_{i} \cup \{v\} (0 \leq i \leq 2)$ for some $v \in G \setminus G_{i}$. Moreover, $|G_{0}|$ = 16, $|G_{1}|$ = 17, $|G_{2}|$ = 18, $|G_{3}|$ = $|G|$ = 19 and $\diam(G_{i})$ = 7 for $0 \leq i \leq3$. Note that $\xi^{c}(G_{0})$ = 182, $\xi^{c}(G_{1})$ = 195, $\xi^{c}(G_{2})$ = 204 and $\xi^{c}(G_{3})$ = $\xi^{c}(G)$ = 213 (The readers can calculate it using the ordered pair at each vertex in $G_{i}$). Using (\ref{eci:vnd}), it is easy to calculate that $\xi^{c}(V_{16,7})$ = 146, $\xi^{c}(V_{17,7})$ = 155,  $\xi^{c}(V_{18,7})$ = 164 and $\xi^{c}(V_{19,7})$ = 173. It is clear from above that $\xi^{c}(G_{0}) \geq \xi^{c}(V_{16,7})$, $\xi^{c}(G_{1}) \geq \xi^{c}(V_{17,7})$, $\xi^{c}(G_{2}) \geq \xi^{c}(V_{18,7})$ and $\xi^{c}(G_{3}) = \xi^{c}(G) \geq \xi^{c}(V_{19,7})$.

\begin{figure}[h!]
\includegraphics[width=5.5cm]{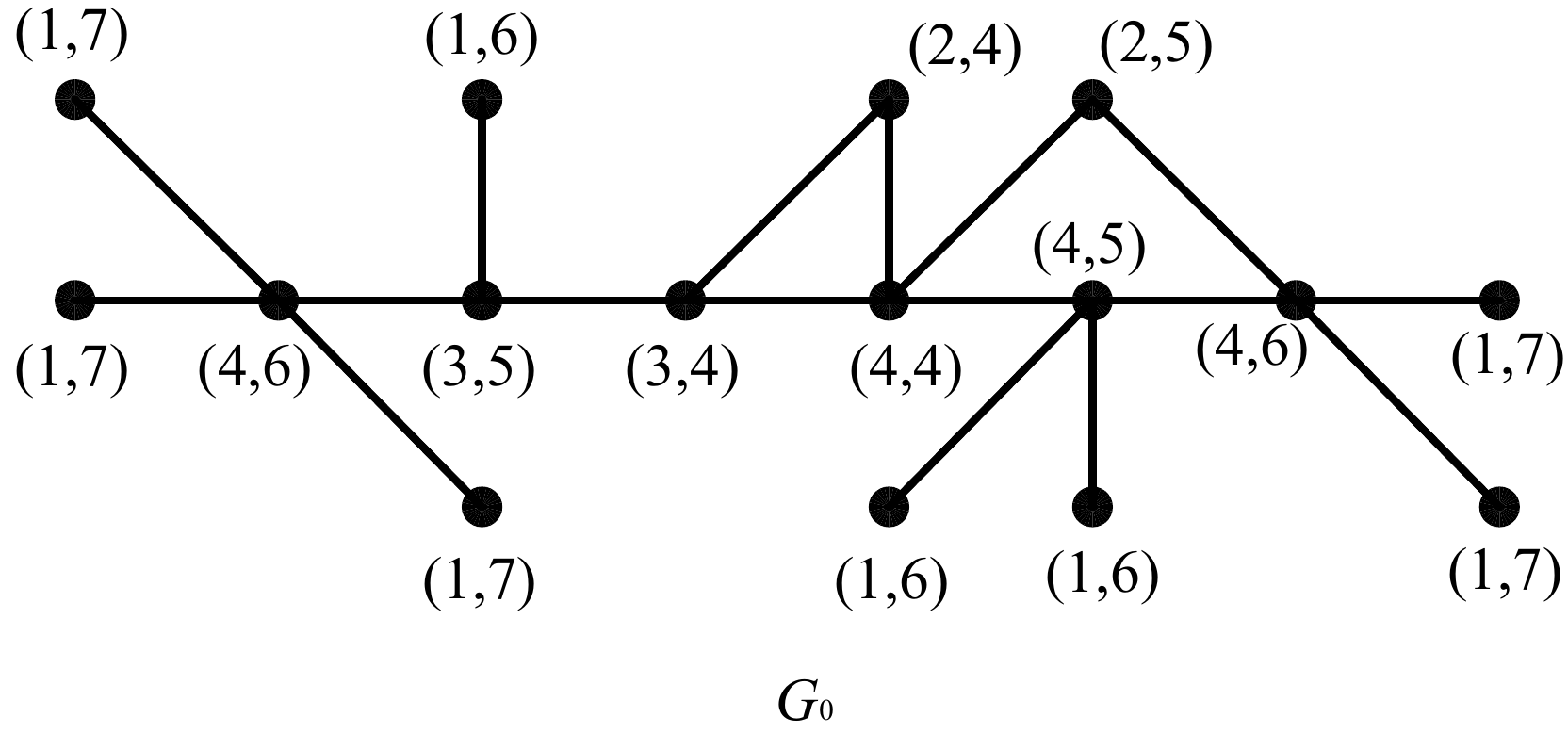}\hfill
\includegraphics[width=5.5cm]{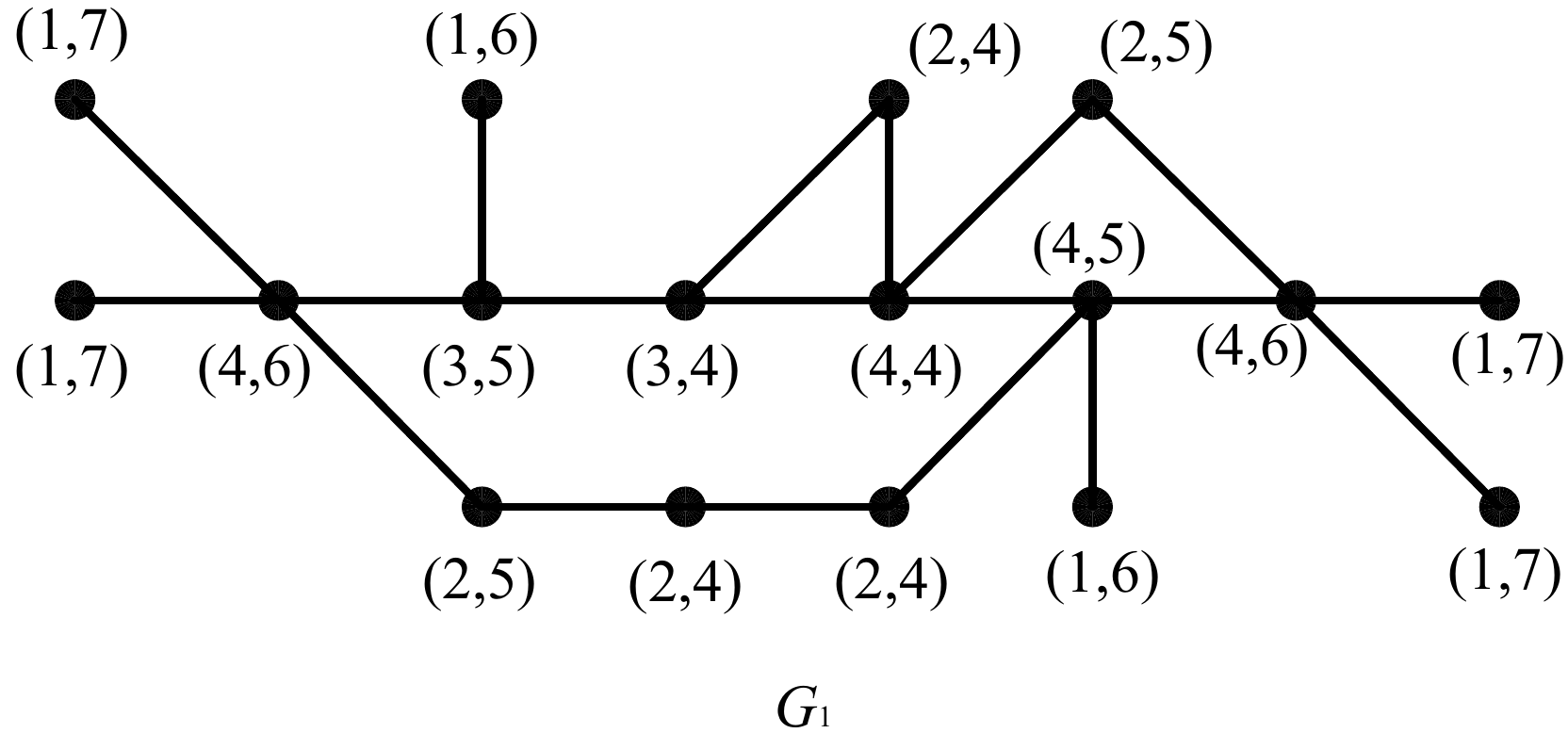} \\ \\
\includegraphics[width=5.5cm]{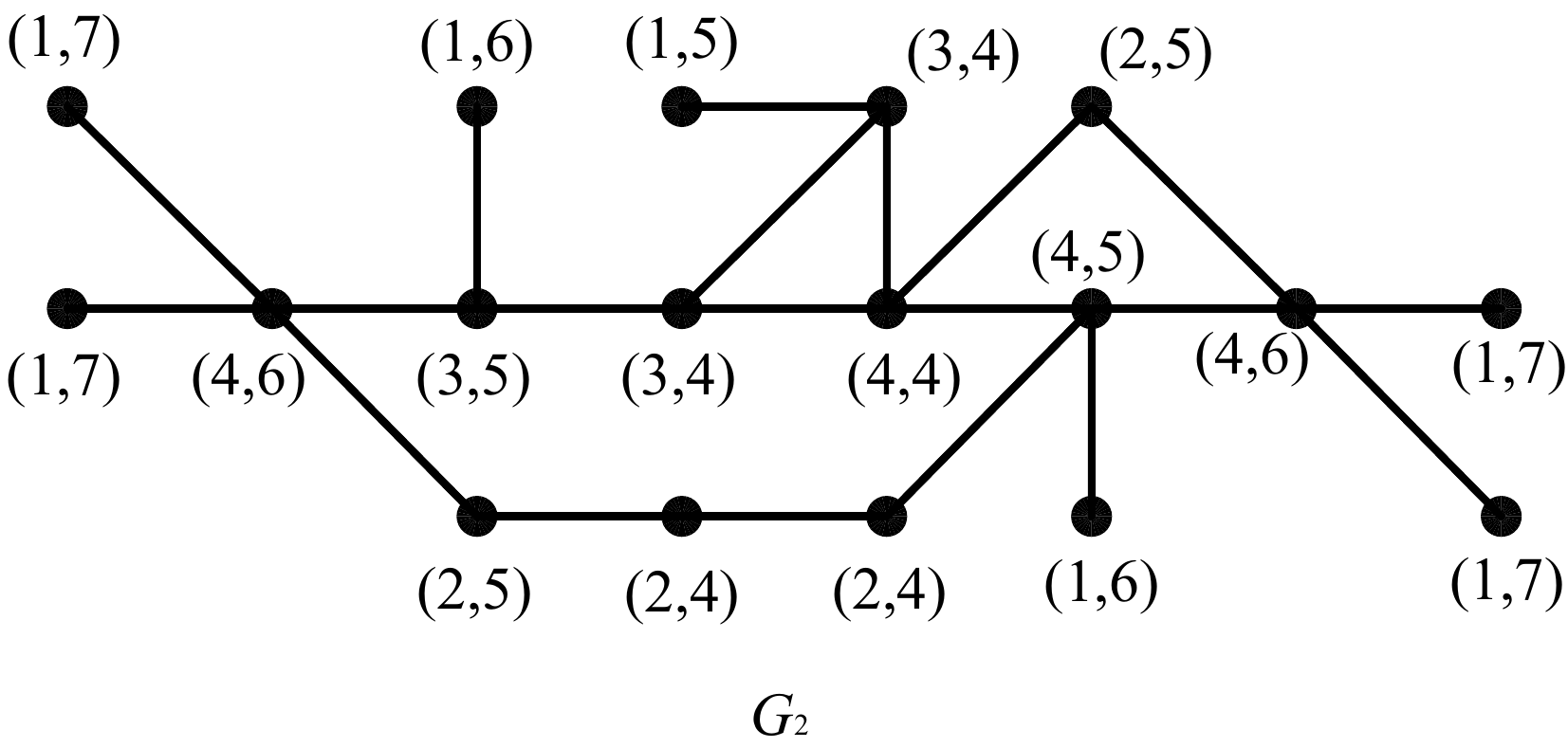}\hfill
\includegraphics[width=5.5cm]{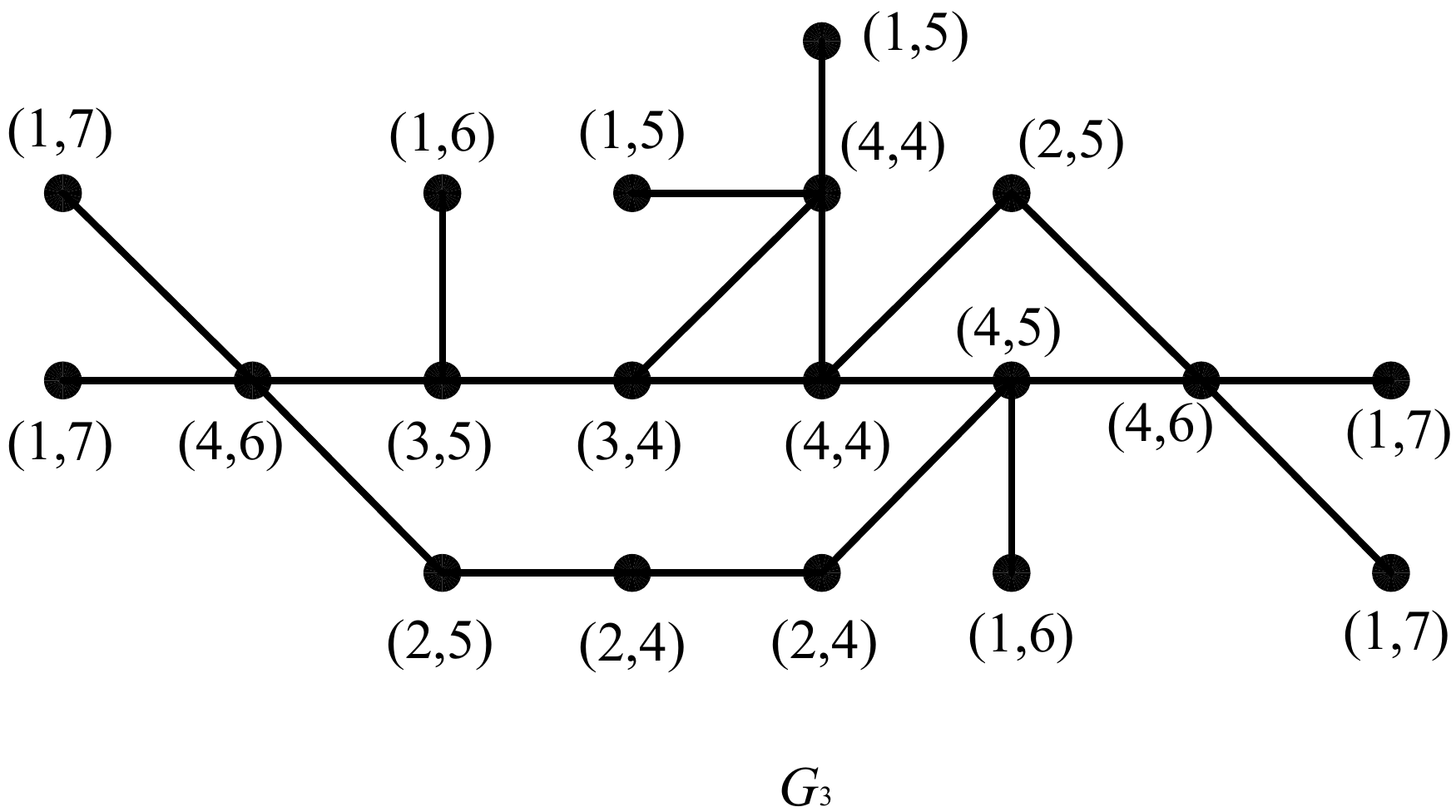}
\caption{Graphs $G_{0} \subset G_{1} \subset G_{2} \subset G_{3} = G$.}\label{Figure2}
\end{figure}

\section*{Concluding remarks}
The determination of extremal graphs for topological indices is remain interest of many researchers due to its various application in chemical, pharmaceutical and other properties of molecules. Various approaches are followed by researchers to determine the extremal graphs for the topological indices and its variants. In this work, we considered the adjacency relation of vertices to determine a lower bound for the eccentric connectivity index of graphs. This approach can also be useful to determine extremal graphs for other topological indices and its variants.

\section*{Acknowledgements}
I want to express my deep gratitude to anonymous referees for kind comments and constructive suggestions.

\end{document}